\documentclass[12pt]{amsart}

\usepackage{amssymb}

\usepackage{enumitem}

\makeatletter
\@namedef{subjclassname@2010}{%
  \textup{2010} Mathematics Subject Classification}
\makeatother

\usepackage[T1]{fontenc}

\newtheorem{theorem}{Theorem}[section]

\newtheorem{lemma}[theorem]{Lemma}

\theoremstyle{definition}

\def\S{{\mathcal S}}
\def\U(#1,#2){\mathop{\mathcal U}(#1,#2)}
\def\In(#1){\mathop{\rm In}(#1)}
\def\Ch(#1){\mathop{\rm Ch}(#1)}
\newcommand*{\atfrac}[2]{\genfrac{}{}{0pt}{}{#1}{#2}}
\newcount\m
\def\n{\the\m\global\advance\m by 1}
\begin{document}

\baselineskip=17pt

\title{Old recurrence formulae for growth series of Coxeter groups}

\author{ Jan Dymara (Wroc\l aw)}

\begin{abstract}
Several classical formulae for the growth series of a Coxeter group
are proved in a new way, using the structure of the Coxeter complex,
the Davis complex, or the Tits non-complex.
\end{abstract}



\maketitle

\section{Introduction}

Throughout this paper $(W,S)$ is a Coxeter system. This means that $W$ is the group generated by the finite set $S$,
subject to relations of the form $(st)^{m_{st}}=1$, where $m_{ss}=1$ and for $t\ne s$ we have
$m_{st}=m_{ts}\in\{2,3,\ldots,\infty\}$ ($\infty$ means no relation). We will usually say
``Coxeter group $W$'' instead of mentioning the whole system. We denote by $\ell(w)$ the word
length of an element $w\in W$ with respect to the generating set $S$. For any subset $T\subseteq S$
the subgroup $W_T$ of $W$ generated by $T$ is also a Coxeter group, with relations being the
relevant relations of $W$. The word length in $W_T$ agrees with the restriction to $W_T$ of
the word length of $W$; we denote both by  $\ell$. The growth series of $W$ is the formal
power series $W(t)=\sum_{w\in W}t^{\ell(w)}$. We will prove several well-known formulae
for this function. For infinite $W$:
\begin{equation}
  \sum_{T\subseteq S}\frac{(-1)^{|T|}}{W_T(t)}=0.
\end{equation}
If $W$ is finite, it has a unique element of longest length $m$. Then
\begin{equation}
\sum_{T\subseteq S}\frac{(-1)^{|T|}}{W_T(t)}=\frac{t^m}{W(t)}.\end{equation}
A subset $T\subseteq S$ is called spherical if $W_T$ is finite. We denote
by $\S$ the set of all spherical subsets of $S$. Let $\chi_T=\sum_{T\subseteq U\in \S}(-1)^{|U|}$.
Then
\begin{equation}\sum_{T\in\S}\frac{(-1)^{|T|}\chi_T}{W_T(t)}=\frac{1}{W(t)};\end{equation}
\begin{equation}\sum_{T\in\S}\frac{(-1)^{|T|}}{W_T(t)}=\frac{1}{W(t^{-1})}.\end{equation}

These formulae are classical and quite popular, cf.~[Bou, Ex.~IV.26], [D, Chapter Seventeen], [ChD];
variants of Coxeter group growth series are 
investigated even nowadays, cf.~[OS], [BGM, Prop.~8.3]. 
The usual proofs use an inclusion--exclusion principle for appropriate
subsets of $W$, or induction on the cardinality of $S$.
Our goal is to use spaces on which $W$ acts to establish (1), (2), (3) and (4)
in an almost uniform way.

\section{Some properties of Coxeter groups}

The basic reference is [Bou], a newer one [D].

\begin{enumerate}[label=$\bullet$]
\item{For any $T\subseteq S$, and any $w\in W$, the coset $wW_T$
has a unique shortest element, say $u$. Moreover, if $w=uw'$, then
$\ell(w)=\ell(u)+\ell(w')$. ([D, Lemma 4.3.1])}
\item{If $W$ is finite, then it has a unique element $w_0$ of longest length
$m=\ell(w_0)$. Then, for any $w\in W$ we have $\ell(w)+\ell(w^{-1}w_0)=\ell(w_0)$
([D, Lemma 4.6.1]). Consequently, in $W$ there are as many elements of length $k$
as of length $m-k$. More succinctly, $W(t)=t^mW(t^{-1})$.}
\item{For any $w\in W$, the set $\In(w)=\{s\in S\mid \ell(ws)<\ell(w)\}$ is spherical ([D, Lemma 4.7.2]).}
\end{enumerate}

\section{Basic spaces}

The basic construction ([D], Chapter Five) is performed as follows.
Let $Y$ be a space (topological, or a simplicial complex) with a collection of subspaces $(Y_s)_{s\in S}$ (called panels).
For $y\in Y$ we put $S(y)=\{s\in S\mid y\in Y_s\}$. Then $\U(W,Y)=W\times Y/\sim$, where
$(w,y)\sim (w',y')\iff y=y'$ and $w^{-1}w'\in W_{S(y)}$. The image of $\{w\}\times Y$ in 
$\U(W,Y)$ will be called a chamber and denoted $wY$. The set of all chambers, $\Ch(\U(W,Y))$,
is in bijective correspondence with $W$, provided $Y\ne \bigcup_{s\in S}Y_s$.
Via this bijection the length function is transported to the set of chambers: $\ell(wY)=\ell(w)$.
The $W$-action on $\U(W,Y)$ is defined by $w(u,y)=(wu,y)$. For $T\subseteq S$ we put
$Y_T=\bigcap_{t\in T}Y_t$ (we call it the face of type $T$) and $Y^T=\bigcup_{t\in T}Y_t$;
this notation and terminology equivariantly extends to all chambers.
We will use the notation $\sigma<X$ to indicate that $\sigma$ is a simplex of a simplicial complex $X$. 

\smallskip
We will use three spaces arising from this construction.

\smallskip
1. The Coxeter complex $X_\Delta=\U(W,\Delta)$, where $\Delta$ is an $(|S|-1)$-dimensional
simplex, and $(\Delta_s)_{s\in S}$ is the collection of its codimension-1 faces.

\smallskip
2. The Davis complex $X_D=\U(W,D)$. The chamber model $D$ (the Davis chamber) is the subcomplex
of the first barycentric subdivision $\Delta'$ of $\Delta$ spanned by the barycentres of faces
of $\Delta$ that are of  spherical type (for finite $W$ one should include also the barycentre of the
empty face---we shall not consider this case).
Alternatively, it is the subcomplex $X_D$ of the first barycentric subdivision $X'_\Delta$ spanned by barycentres
of simplices of $X_\Delta$ that are contained in finitely many chambers. It is a locally finite
complex. The panels are $D_s=D\cap\Delta_s$, and the faces $D_T=D\cap\Delta_T$.
An important property: for spherical $T$ the face $D_T$ is contractible; indeed, it is
a cone with apex at the barycentre of $\Delta_T$. For non-spherical $T$ we have $D_T=\emptyset$.
For any simplex $\sigma<D$ we define its type
$S(\sigma)=\{s\in S\mid \sigma<D_s\}$; this definition extends by invariance to all simplices in $X_D$.

\smallskip
3. The Tits non-complex $X_\Delta^f=\U(W,\Delta^f)$. Here $\Delta^f$ is obtained from $\Delta$ by
removing the closed simplices whose types are non-spherical. It is not a simplicial complex.
For faces $\sigma$ of $\Delta$ of spherical type we still use the notation $\sigma<\Delta^f$,
even though, in the geometric realisation, some part of $\sigma$ may not be a subset of $\Delta^f$.

\smallskip
For infinite $W$, all of the above spaces are contractible;
$X_D$ and $X_\Delta^f$ even carry good CAT(0) metrics. Though they are
generally very useful, these properties will not be relied upon in this note.

\section{Euler characteristic generating function}

Let $X$ be either the Coxeter complex or the Davis complex for the Coxeter system $(W,S)$.
(The case of the Tits non-complex is slightly different and will be dealt with
in the last section.)
For a simplex $\sigma<X$ we define
its length $\ell(\sigma)$ as the minimum of $\ell(C)$, where $C$ runs through the chambers of $X$
that contain $\sigma$. There are finitely many chambers $C$ of $X$ of any given length,
and each of them contains finitely many simplices. Therefore the following formal
power series---the Euler characteristic generating function of $X$---is well defined:
\begin{equation}\chi_t(X)=\sum_{\sigma<X}(-1)^{\dim{\sigma}}t^{\ell(\sigma)}.\end{equation}
We will calculate this series in two ways: grouping simplices according to type,
and by local summation in $X$. Comparing the results we will obtain the formulae (1),(2) and (3).

\section{Summation according to type}

We consider $X=\U(W,Y)$ for $Y=\Delta$ or $D$. 
Let $\sigma<X$ be a simplex of type $T$. Then the chambers $C$ that contain $\sigma$ correspond
to elements of some coset $uW_T$. Assume that $u$ is the shortest element in that coset.
Then $\ell(\sigma)=\ell(u)$, and
\begin{equation}\sum_{C>\sigma}t^{\ell(C)}=\sum_{w\in uW_T}t^{\ell(w)}=
\sum_{v\in W_T}t^{\ell(u)}t^{\ell(v)}=t^{\ell(u)}W_T(t)=t^{\ell(\sigma)}W_T(t).\end{equation}
Multiplying both sides by $(-1)^{\dim{\sigma}}$ and summing over all simplices of type $T$ we get
\begin{equation}\sum_{C\in\Ch(X)}t^{\ell(C)}\sum_{\atfrac{\sigma<C}{S(\sigma)=T}}(-1)^{\dim{\sigma}}=
W_T(t)\sum_{\atfrac{\sigma<X}{S(\sigma)=T}}(-1)^{\dim\sigma}t^{\ell(\sigma)},\end{equation}
which implies
\begin{equation}
\frac{W(t)}{W_T(t)}\sum_{\atfrac{\sigma<Y}{S(\sigma)=T}}(-1)^{\dim{\sigma}}
=\sum_{\atfrac{\sigma<X}{S(\sigma)=T}}(-1)^{\dim\sigma}t^{\ell(\sigma)}.
\end{equation}
Finally, we sum over all $T\subseteq S$:
\begin{equation}W(t)\sum_{T\subseteq S}\frac{1}{W_T(t)}\sum_{\atfrac{\sigma<Y}{S(\sigma)=T}}(-1)^{\dim{\sigma}}=\chi_t(X).\end{equation}
We will calculate the inner sum separately in each case.

\smallskip
The Coxeter complex. For any proper subset $T\subset S$ there is one $\sigma<\Delta$ of type $T$, namely $\Delta_T$.
It has dimension ${|S|-|T|-1}$. Thus
\begin{equation}\chi_t(X_\Delta)=W(t)\sum_{T\subset S}\frac{(-1)^{|S|-|T|-1}}{W_T(t)}.\end{equation}

The Davis complex. Recall that the face $D_T$ is a cone with apex at the barycentre of $\Delta_T$;
let $L_T'$ be its base $D_T\cap\partial\Delta_T$. The simplices $\sigma<D$ of type $T$
are the interior simplices of this cone (including the apex), so that
\begin{equation}\sum_{\atfrac{\sigma<D}{S(\sigma)=T}}(-1)^{\dim{\sigma}}=1-\chi(L_T').\end{equation}
Simplices in $L_T'$ correspond to chains $U_1\subset\ldots\subset U_k$, where $U_i$
are spherical and properly contain $T$. This means that $L_T'$ is, as the notation suggests,
the barycentric subdivision of the following simplicial complex $L_T$. The vertices of
$L_T$ correspond to spherical sets of the form $T\cup\{u\}$ (with $u\in S-T$);
vertices $T\cup\{u_1\},\ldots,T\cup\{u_k\}$ span a simplex if $T\cup\{u_1,\ldots,u_k\}$
is spherical. Thus, simplices of $L_T$ correspond to spherical $U$ containing $T$, and
\begin{equation}1-\chi(L_T')=1-\chi(L_T)=\sum_{T\subseteq U\in\S}(-1)^{|U|-|T|}=(-1)^{|T|}\chi_T\end{equation}
for $\chi_T=\sum_{T\subseteq U\in\S}(-1)^{|U|}$. Finally, returning with this to (9) we get
\begin{equation}\chi_t(X_D)=W(t)\sum_{T\in\S}\frac{(-1)^{|T|}\chi_T}{W_T(t)}.\end{equation}


\section{Local summation}

We continue with the cases $Y=\Delta$ or $D$. 
Now we compute the sum $\sum_{\sigma<X}(-1)^{\dim{\sigma}}t^{\ell(\sigma)}$ by arranging its summands according to
chambers: we consider each $\sigma$ as part of the chamber realising $\ell(\sigma)$. We get
\begin{equation}\chi_t(X)=\sum_{C\in \Ch(X)}t^{\ell(C)}\sum_{\atfrac{\sigma<C}{\ell(\sigma)=\ell(C)}}(-1)^{\dim{\sigma}}
=\sum_{w\in W}t^{\ell(w)}\sum_{\atfrac{\sigma<wY}{\ell(\sigma)=\ell(w)}}(-1)^{\dim{\sigma}}.\end{equation}

\begin{lemma}
Let $\sigma<wY$. Then $\ell(\sigma)<\ell(w)$ if and only if $\sigma<(wY)^{\In(w)}$.
\end{lemma}

\begin{proof}
Let $T$ be the type of $\sigma$. The chambers containing $\sigma$ correspond
to elements of the coset $wW_T$.

Suppose $\ell(\sigma)<\ell(w)$. It follows that $w$ is not the shortest element of the coset $wW_T$.
Therefore $\ell(wt)<\ell(w)$ for some $t\in T$. This $t$ also belongs to  $\In(w)$.
Finally, $\sigma<wY\cap wtY=(wY)_t$.

Conversely, if $\sigma<(wY)^{\In(w)}$, then $\sigma<(wY)_s$ for some $s\in\In(w)$.
But then $\sigma<wsY$ and $\ell(\sigma)\le\ell(ws)<\ell(w)$. The Lemma is proved.
\end{proof}

By the Lemma, the inner sum in (14) is performed over simplices of $wY$ that are not in $(wY)^{\In(w)}$.
Thus
\begin{equation}
\begin{aligned}
\sum_{\atfrac{\sigma<wY}{\ell(\sigma)=\ell(w)}}(-1)^{\dim{\sigma}}&=
\chi(wY,(wY)^{\In(w)})=\chi(Y,Y^{\In(w)})\\
&= \chi(Y)-\chi(Y^{\In(w)})=1-\chi(Y^{\In(w)}),
\end{aligned}
\end{equation}
the last equality due to contractibility of $Y$.

We split the calculation of $\chi(Y^{\In(w)})$ into three cases.

\smallskip
\noindent 1) The generic case: $\In(w)$ is a proper non-empty subset of $S$, of some cardinality $k$.
Notice that for any $T\subseteq \In(w)$ (a fortiori, spherical)
the face $Y_T$ is contractible. Using the inclusion--exclusion principle
we get
\begin{equation}
\begin{aligned}
  \chi(&Y^{\In(w)})=\sum_{s\in\In(w)}\chi(Y_s)-\sum_{\atfrac{s,t\in\In(w)}{s\ne t}}\chi(Y_{\{s,t\}})+\ldots+(-1)^{k-1}\chi(Y_{\In(w)})\\
  &=k-\binom{k}{2}+\binom{k}{3}-\ldots+(-1)^{k-1}\binom{k}{k}=1-(1-1)^k=1.
\end{aligned}
\end{equation}
\noindent 2) The set $\In(w)$ is empty. Then $\chi(Y^{\In(w)})=\chi(\emptyset)=0$.
This happens exactly when $w=1$.

\smallskip
\noindent 3) The case $\In(w)=S$. This happens only if $W$ is finite and $w$ is the element of longest length (say length $m$).
Then $Y^S$ is a triangulation of the $(|S|-2)$-dimensional sphere, $\chi(Y^S)=1-(-1)^{|S|-1}$.

\smallskip
Plugging these results into (15), and then into (14), we get for infinite $W$
\begin{equation}\chi_t(X_\Delta)=1,\qquad \chi_t(X_D)=1,\end{equation}
while for finite $W$ with element of longest length $m$
\begin{equation}\chi_t(X_\Delta)=1+(-1)^{|S|-1}t^m.\end{equation}
Comparing these results with (10) and (13) we get (1), (2) and (3).

\section{The Tits non-complex}

Consider now the case $X=X_\Delta^f$. For $\sigma<X$  we define $L(\sigma)$ 
as the maximum of $\ell(C)$ over all chambers containing $\sigma$. This is finite, since the type of
$\sigma$ is spherical. Then we define $\chi^t(X)=\sum_{\sigma<X}(-1)^{\dim{\sigma}}t^{L(\sigma)}$.
As before, we will calculate this sum in two ways.

First, we group the summands according to the type of $\sigma$.
Let $\sigma<X$ be a simplex of (spherical!) type $T$. Then the chambers $C$ that contain $\sigma$ correspond
to elements of some finite coset $uW_T$. Assume that $u$ is the longest element in that coset.
Then $L(\sigma)=\ell(u)$, and
\begin{equation}
\begin{aligned}
  \sum_{C>\sigma}t^{\ell(C)}=\sum_{w\in uW_T}t^{\ell(w)}=\sum_{v\in W_T}t^{\ell(u)}t^{-\ell(v)}&=t^{\ell(u)}W_T(t^{-1})\\
  &=t^{L(\sigma)}W_T(t^{-1}).
\end{aligned}
\end{equation}
Multiplying both sides by $(-1)^{\dim{\sigma}}$($=(-1)^{|S|-|T|-1}$) and summing over all simplices of type $T$ (there is just one in each chamber) we get
\begin{equation}(-1)^{|S|-|T|-1}\sum_{C\in\Ch(X)}t^{\ell(C)}=
W_T(t^{-1})\sum_{\atfrac{\sigma<X}{S(\sigma)=T}}(-1)^{\dim\sigma}t^{L(\sigma)},\end{equation}
which implies
\begin{equation}
\frac{W(t)}{W_T(t^{-1})}(-1)^{|S|-|T|-1}
=\sum_{\atfrac{\sigma<X}{S(\sigma)=T}}(-1)^{\dim\sigma}t^{L(\sigma)}.
\end{equation}
Summing this equality over all spherical types $T$ we get
\begin{equation}W(t)\sum_{T\in \S}\frac{(-1)^{|S|-|T|-1}}{W_T(t^{-1})}=\chi^t(X^f_\Delta).\end{equation}
\smallskip

Second, we group summands in chambers according to $L$:
\begin{equation}\chi^t(X)=\sum_{\sigma<X}(-1)^{\dim{\sigma}}t^{L(\sigma)}=
\sum_{w\in W}t^{\ell(w)}\sum_{\atfrac{\sigma<w\Delta^f}{L(\sigma)=\ell(w)}}(-1)^{\dim{\sigma}}.\end{equation}
Every $\sigma<w\Delta^f$ is of the form $w\Delta_T$ for some spherical $T$; moreover,
$L(\sigma)=\ell(w)$ if and only if $w$ is the longest element in the coset $wW_T$, which
happens exactly when $T\subseteq \In(w)$. Therefore
\begin{equation}\sum_{\atfrac{\sigma<w\Delta^f}{L(\sigma)=\ell(w)}}(-1)^{\dim{\sigma}}=\sum_{T\subseteq\In(w)}(-1)^{|S|-|T|-1}=(-1)^{|S|-1}(1-1)^{|\In(w)|}.\end{equation}
Thus, the only non-zero summand corresponds to $w$ with empty $\In(w)$, that is to $w=1$.
We get \begin{equation}\chi^t(X^f_\Delta)=(-1)^{|S|-1}.\end{equation} 
Comparing (25) with (22), and switching $t$ and $t^{-1}$, we get (4).

\subsection*{Acknowledgements}

The author was partially supported by the Polish National Science Centre
(NCN) grant 2016/23/B/ST1/01556.

\address{

Instytut Matematyczny

Uniwersytet Wroc\l awski

pl.~Grunwaldzki 2/4,

50-384 Wroc\l aw

Poland}

\email{dymara@math.uni.wroc.pl}


\normalsize

\begin{thebibliography}{BGM}

\baselineskip=17pt

\bibitem[Bou]{Bou} N.~Bourbaki. 
\emph{Groupes et alg\`ebres de Lie, Chapitres IV--VI.} Masson, 2nd edition, 1981.

\bibitem[BGM]{BGM} M.~Bo\. zejko, \' S.~Gal, W.~M\l otkowski.
\emph{Positive definite functions on Coxeter groups
with applications to operator spaces and noncommutative probability},
Comm.~Math.~Phys.~{\bf 361}(2018), no.~2, 583--604.

\bibitem[ChD]{ChD} R.~Charney, M.~Davis. 
\emph{Reciprocity of growth functions of Coxeter groups}.
Geom.~Dedicata {\bf 39}(1991), no.~3, 373--378.


\bibitem[D]{D} M.~Davis.
\emph{The Geometry and Topology of Coxeter Groups}.
\rm Princeton University Press, Princeton and Oxford 2008.

\bibitem[OS]{OS} B.~Okun, R.~Scott.
\emph{Growth series of CAT(0) cubical complexes}, 
Topology Proc.~{\bf 54}(2019), 295--303.

\end{thebibliography}
\end{document}